\documentclass[english,12pt,oneside]{amsproc}
\usepackage[english]{babel}
\usepackage{amsthm}
\usepackage{graphics}
\usepackage{amsfonts, amssymb, amscd, amsmath}
\usepackage{latexsym}
\usepackage[matrix,arrow,curve]{xy}
\usepackage{graphicx}
\usepackage{tikz}
\allowdisplaybreaks[1]

\DeclareMathOperator{\Hom}{Hom} 

\DeclareMathOperator{\diag}{diag}

\DeclareMathOperator{\Mat}{Mat}
\DeclareMathOperator{\Herm}{Herm}
\DeclareMathOperator{\Spec}{Spec}
\DeclareMathOperator{\Int}{Int}
\DeclareMathOperator{\Hess}{Hess}

\newcounter{stmcounter}[section]
\newcounter{thcounter}

\numberwithin{equation}{section}

\theoremstyle{plain}

\newtheorem{thm}[thcounter]{Theorem}

\newtheorem{prop}[stmcounter]{Proposition}
\newtheorem{lem}[stmcounter]{Lemma}

\theoremstyle{definition}
\newtheorem{defin}[stmcounter]{Definition}

\theoremstyle{remark}
\newtheorem{ex}[stmcounter]{Example}
\newtheorem{rem}[stmcounter]{Remark}

\begin{document}

\title{Orbit spaces of torus actions on Hessenberg varieties}

\author[V.\,V.~Cherepanov]{V.\,V.~Cherepanov}
\address{Faculty of Computer Science, Higher School of Economics, Moscow, Russia}
\email{vilamsenton@gmail.com}

\date{\today}
\thanks{This work is supported by the Russian Science Foundation under grant 18-71-00009.}

%\subjclass[2010]{Primary 55R91, 57S15, 57S25, 22F30; Secondary 57R91, 57S17, 20G20, 57M60, 20G41}


\begin{abstract}
We consider effective actions of a compact torus $T^{n-1}$ on an even-dimensional smooth manifold $M^{2n}$ with isolated fixed points. We prove that under certain conditions on weights of tangent representations, the orbit space is a manifold with corners. Given that the action is Hamiltonian, the orbit space is homeomorphic to $S^{n+1} \setminus (U_1 \sqcup \ldots \sqcup U_l)$ where $S^{n+1}$ is the $(n+1)$--sphere and $U_1, \ldots, U_l$ are open domains. We apply the results to regular Hessenberg varieties and manifolds of isospectral Hermitian matrices of staircase form. 
\end{abstract}

\maketitle

\section{Introduction}

	Let $T^k$ be a compact torus effectively acting on a smooth compact orientable even-dimensional manifold $M^{2n}$ with isolated fixed points. We denote by $d = n - k$ the complexity of the action. Actions of complexity 0 are studied in toric topology (see \cite{BP1}, \cite{BP2}) while actions of positive complexity are mostly considered from the point of view algebraic and symplectic geometry.
	
	V. Buchstaber and S. Terzic (see \cite{BT1}, \cite{BT2}, \cite{BT3}) studied the orbit spaces of such torus actions. They introduced the notion of $(2n,k)$--manifolds and obtained topological models for such actions. They showed that for the actions of tori on the complex Grassmann manifold $G_{4,2}$ and the manifold $Fl_3$ of full complex flags, the orbit spaces are homeomorphic to the spheres $S^5$ and $S^4$ correspondingly. In \cite{A1}, A.~Ayzenberg showed that for actions of complexity 1 in general position, the orbit space is a topological manifold.
	
	In the paper, we prove the generalization of the result to the case of actions of complexity 1 not in general position. In the case, the orbit space is a manifold with corners, i.e. a manifold with boundary from a topological point of view. We introduce a natural partition $M^{2n} = \bigsqcup_i V_i$ by orbit types and specify elements $V_i$ such that their orbits are boundary points. Moreover, we prove the similar generalization of the result by Y. Karshon and S. Tolman \cite{KT} regarding Hamiltonian actions of complexity 1. Their method allows us to construct a model for the orbit space using the moment polytope $\mu(M^{2n})$. Note that similar results were obtained in \cite{Su} for smooth varieties.
	
	We consider manifolds $M_h$ of Hermitian matrices of staircase form (which is defined by a Hessenberg function $h$, see \cite{AB1}) with simple spectrum. Each manifold $M_h$ possesses a canonical torus action with nontrivial complexity unless $M_h$ is the manifold of isospectral tridiagonal matrices. The actions are not in general position and this yields that for $M_h$ with the action of complexity 1, the orbit spaces are topological manifolds with boundary. It turns out (see \cite{AB1}), that for each manifold $M_h$, there is a "twin" $H_h$ --- a regular semisimple Hessenberg variety defined by \nolinebreak $h$. The orbit spaces $M_h / T^k$ and $H_h / T^k$ coincide (again, see \cite{AB1}). We explicitly describe topology of the orbit spaces of $M_h$ in the case of $4 \times 4$ and $5 \times 5$ matrices and actions of complexity $d = 1$; it corresponds to the case of Hessenberg varieties $H_h$ in $\mathbb{C}^4$ and $\mathbb{C}^5$.
	
	The author expresses his great gratitude to Victor Buchstaber and Anton Ayzenberg for their guidance and help during the work; the author expresses his sincere thanks to Mikiya Masuda and Nikita Klemyatin for valuable discussions and support.
	
	\section{Partitions by orbit types and actions of complexity 1}
	
	Consider an effective action of the compact torus $T = T^k$ on a smooth compact manifold $M = M^{2n}$.
	
	Define the partition by orbit types $$M = \bigsqcup_i V_{G, i}$$
	
	where $V_{G, i}$ denotes a connected component of the set $M^G = \{x \in M: St(x) = \nolinebreak G\}$ of points $x$ with common stabilizers.
	
	As in \cite{A1}, we consider actions satisfying the following (*) conditions:
	
	$\bullet$ The set of fixed points $M^T$ is finite and each fixed point is isolated.
	
	$\bullet$ (Connected stabilizers) For any point $x \in M^{2n}$, its stabilizer subgroup $St(x) \subset \nolinebreak T^k$ is a torus $T_x \subset T^k$.
	
	$\bullet$ (Adjoint orbits) The closure of every partition element $\overline{V}_{G, i}$, $G \subset T$, contains a point $x'$ with  $\dim(T_{x'}) > \dim(G)$ (unless $V_{G, i}$ is a fixed point).
	
	%Мотивация для изучения фильтрации по типу орбит и пространства орбит $Q = M / T$ состоит в следующем: это позволяет нам построить топологическую модель для исходного многообразия $M$. То есть, имея отображение $\chi: Q \longrightarrow S(T)$ в множество $S(T)$ замкнутых связных подгрупп тора $T$, которое орбите $Tx \in Q$ ставит в соответствие её стабилизатор $T_x \in S(T)$, мы построим модель:
	
	%$$M \cong (Q \times T) / ~$$
	
	%-- эквивариантный гомеоморфизм. Отметим, что для построения такой модели существенно потребовать, чтобы каноническая проекция $p: M \longrightarrow Q$ имела сечения. \\
	
	Let us state the result by A. Ayzenberg (see \cite{A1}) concerning actions of complexity $d = 1$, i.e. actions of $T^{n-1}$ on $M^{2n}$.
	
	Let $x \in M^T$ be a fixed point. We have the tangent representation of $T^{n-1}$ at $x$ and let $\alpha_1, \ldots, \alpha_n \in \Hom(T^{n-1}, S^1)$ be its weights, i.e. $$T_x (M^{2n}) \cong V(\alpha_1) \otimes \ldots \otimes V(\alpha_n)$$
	where $V(\alpha_i)$ is the standard 1-dimensional complex representation given by $t[z] = \alpha_i(t) z$ for $t \in T^{n-1}$ and $z \in \mathbb{C}$.
	
	Since $\Hom(T^{n-1}, S^1) \cong \mathbb{Z}^{n-1}$, we have $n$ integer vectors $\alpha_1, \ldots, \alpha_n$ in $(n-1)$-dimensional space. Hence there is a relation:
	$$c_1 \alpha_1 + \ldots + c_n \alpha_n = 0$$
	--- for some $c_i \in \mathbb{Z}$ and we can suppose that $gcd(c_1, \ldots, c_n) = 1$. The weights are \textit{in general position} if $c_i \neq 0$ for $i = 1, \ldots, n$ --- or, equivalently, if every $n-1$ of $n$ weights are linearly independent (more generally: $n$ vectors in $\mathbb{Z}^k$ are in general position if every $k$ of them are linearly independent). An action $T^{n-1}$ on $M^{2n}$ is called an action in general position if weights of its tangent representation at every fixed point are in general position.
	
	\begin{thm}{\cite{A1}}
		Assume that an action of $T^{n-1}$ on $M^{2n}$ (subject to (*)) is in general position. Then the orbit space $Q = M^{2n} / T^{n-1}$ is a topological manifold.
	\end{thm}
	
	%В следующем параграфе будет доказано обобщение этой теоремы на случай, когда в каких-то точках $x$ весовые вектора не находятся в общем положении. В этом случае оказывается, что пространство орбит $Q$ также является топологическим многообразием, но с краем -- окрестность таких неподвижных точек $x$ и будет представлять в факторе границу $\partial(Q)$. Притом можно явно указать, какие элементы фильтрации, близкие к $x$, будут содержаться в границе $\partial(Q)$. 
	
	In \cite{BT1} and \cite{BT2}, there were considered natural actions of the torus $T^3$ on complex Grassmann manifold $G_{4,2}$ and the torus $T^2$ on the complex full flag manifold $Fl_3$. Using the theory of $(2n,k)$--manifolds, the authors proved that the orbit spaces of the actions are topological spheres, i.e.
	$$G_{4,2} / T^3 \cong S^5,$$
	$$Fl_3 / T^2 \cong S^4.$$
	Assume that the $T$--action on $M$ is not in general position: there exists at least one fixed point $x \in M^T$ for which the weights of tangent representation $\alpha_1, \ldots, \alpha_n$ are not in general position. In the case, the following holds true:
	
	%$$M^G = \{x \in M: St(x) = G \subset T \}$$
	
	%-- её замыкание содержит некоторую точку $x':$ dim $St(x') >$ dim $G$.
	
	%Теперь отметим, что в каждой неподвижной точке $x \in M^T$ мы имеем представление тора в $T_x M$ -- касательное представление. Оно разлагается в сумму неприводимых одномерных комплексных представлений:
	
	%$$
	%T_x M \cong V(\alpha_1) \oplus \ldots \oplus V(\alpha_n)
	%$$
	
	% -- где $\alpha_i \in Hom(T^{n-1}, S^1)$ -- весовые вектора, а $V(\alpha_i)$ -- стандартное одномерное представление тора: $t [x] = \alpha_i (t) [x]$. Заметим, что в случае, когда на $M$ есть комплексная структура, знаки $\alpha_i$ согласовываются однозначно, однако для дальнейших рассуждений мы не требуем однозначной определенности знаков.
	
	%Действие тора $T^{n-1}$ на $M^{2n}$ называется действием в общем положении, если в каждой неподвижной точке $x \in M^T$ весовые вектора $\alpha_1 \ldots \alpha_n$ находятся в общем положении -- то есть любые $n-1$ весовых вектора линейно независимы. 
	
	\begin{thm}
		The orbit space $Q = M / T$ is a manifold with corners (i.e. a manifold with boundary).
	\end{thm}
	
	\begin{proof}
		Following the proof for actions in general position (see \cite{A1}), we  prove the local statement.
		\begin{lem}
			Let $x \in M^T$ be a fixed point for which the weights $\alpha_i$ of the tangent representation are not in general position. Then there exists a neighborhood $U$, $Tx \in U \subset Q$ such that $U \cong \mathbb{R}^{m+1} \times \mathbb{R_{\geq}}^{n-m}$.
		\end{lem}
		
		\begin{proof}
			Since the point $x \in M$ possesses a neighborhood equivariantly \linebreak  homeomorphic to $\mathbb{C}^n$, it is sufficient to prove the statement for a subtorus $T = T^{n-1} \subset G = T^{n}$ where $G$ acts on $\mathbb{C}^n$ standardly. Denote $U = \mathbb{C}^n / T^{n-1}$; we will show that $U$ is the required neighborhood of $Tx \in Q$.
			
			There exist $c_i \in \mathbb{Z}$ such that $c_1 \alpha_1 + \ldots + c_n \alpha_n = 0$ and gcd$(c_1, \ldots, c_n) = 1$; since the action is not in general position, $c_i = 0$ for some $i$. Without loss of generality, we can assume that $c_i \neq 0$ for $i \in \{1, \ldots, m\}$ and $c_i = 0$ otherwise.
			
			We have
			$$c_1 \alpha_1 + \ldots + c_m \alpha_m = 0 $$
			where each $c_i \neq 0$. The action of the subtorus given by the following equation
			\begin{equation}
			t_1^{c_1} \cdot \ldots \cdot t_m^{c_m} = 1
			\end{equation}
			coincides with the action of the torus $T$ on $\mathbb{C}^n$. We have the following diagram connecting the actions of $T$ and $G$:
			\begin{center}
				\begin{tikzpicture}[scale=1.0]
				
				\coordinate (A1) at (-1,0);
				\coordinate (A2) at (1,-0.07);						
				\coordinate (A3) at (0, -1);
				\coordinate (B1) at (0,0.2);
				\coordinate (B2) at (-0.725,-0.725);
				
				\draw [thick,latex-latex] [->](-0.8,-0.05)--(0.7,-0.05) node[right]{};
				\draw [thick,latex-latex] [->](-0.85,-0.15)--(-0.15,-0.85) node[right]{};
				\draw[densely dashed] [thick,latex-latex] [->](0.15,-0.85)--(0.85,-0.15) node[right]{};
				
				\node at (barycentric cs:A1=1) {$\mathbb{C}^n$};
				\node at (barycentric cs:A2=1) {$\mathbb{R}_{\geq}^n$};
				
				\node at (barycentric cs:A3=1)  {$U$};						
				\node at (barycentric cs:B1=1) {\small $/G$};
				\node at (barycentric cs:B2=1) {\small $/T$};
				
				\end{tikzpicture}
			\end{center}
			where the dashed line denotes taking quotient of $U$ by $G/T = T^1 = S^1$.
			
			Now we note that $T$ contains exactly $n-m$ coordinate circles since equation (2.1) doesn't contain the variables $t_{m+1}, \ldots, t_n$. Therefore, the action of $T$ decomposes into the product of actions on $\mathbb{C}^n = \mathbb{C}^m \times \mathbb{C}^{n-m}$ where the action on the first factor is in general position and the action on the second factor is standard. Thus, we have $U = U' \times \mathbb{R}_{\geq}^{n-m}$ where $U'$ is the orbit space of a $T^{m-1}$-action on $\mathbb{C}^m$ in general position. Therefore, $U' \cong \mathbb{R}^{m+1}$ and $U \cong \mathbb{R}^{m+1} \times \mathbb{R}_{\geq}^{n-m}$. $\hbox{}$
		\end{proof}
		
		Now that we already know there is a neighborhood $U$ of the point $Tx \in Q$ homeomorphic to the half-space $H^{n+1} = \mathbb{R}^n \times \mathbb{R}_{\geq}$, we finish the proof as in the case of actions in general positions. Due to adjoint orbits condition, given a point $x \in V_{H,i}$, there is a point $x' \in \overline{V}_{H,i}$ such that $\dim St(x') > \dim St(x) = \dim H$. For both $x$ and $x'$, the local representation of the stabilizer subgroup $St(x)$ in the normal space $T_x M / T_x Tx$ coincide; therefore, we can prove the theorem by induction on the dimension of stabilizer subgroup $\dim H$.
		
		Note, however, that there exists such element $V_{T_1, i}$ where $T_1$ is the torus from the proof of Lemma 2.3 which acts standardly on $\mathbb{C}^{n-m}$. The tangent space of $V_{T_1, i}$ at the corresponding fixed point $x$ is the space $V(\alpha_1) \oplus \ldots \oplus V(\alpha_m)$. It follows from the proof of Lemma 2.3 that the boundary $\partial Q$ locally looks like $X/T \times \mathbb{R}^{n-m-1}$.
		
		%	Возьмем точку $Tx$ в $Q$, такую что $x \in M^H$. В силу условия присоединенности $\exists x' \in M^{H'}$, которая лежит в замыкании компоненты связности $M^H$, содержащей $x$, притом dim$H'>$ dim$H$. В пространстве орбит $Tx'$ лежит в окрестности $Tx$, притом локальные представления $St(x)$ в нормальных пространствах $T_x M / T_x Tx$ в точках $x$ и $x'$ совпадают.
		
	\end{proof}
	
	Let us make the following remarks. According to Lemma 2.3, if the weight of the tangent representation at $x \in M^T$ are not in general position, then the neighborhood of $Tx$ in the orbit space $Q$ is the corner $\mathbb{R}^{m+1} \times \mathbb{R}_{\geq}^{n-m}$ where $m$ is determined by the weights $\alpha_i$. Therefore, the point $Tx$ belongs to the boundary $\partial Q$. Nonetheless, for actions not in general position, there might be fixed points with weights in general position. For such points, the neighborhood $U \cong \mathbb{R}^{n+1}$, i.e. $Tx$ is an interior point.
	
	Moreover, we might conclude from the structure of the corner $\mathbb{R}^{m+1} \times \mathbb{R}_{\geq}^{n-m}$ which elements $V_{H, i}$ are mapped to the boundary $\partial Q$. The torus $T$ in a local chart acts as $T_0 \times T_1$ on $\mathbb{C}^n = \mathbb{C}^m \times \mathbb{C}^{n-m}$ where the action of $T_1$ on $\mathbb{C}^{n-m}$ is standard. Therefore, $V_i / T \subset \partial Q$ if and only if $H \cap T_1 \neq 0$ (see an example after Remark 3.4). 
	
	Note that the orbit space $Q$ possesses the structure of a manifold with corners. For a point $y \in \partial Q$, its neighborhood is a corner $\mathbb{R}^{m+1} \times \mathbb{R}_{\geq}^{n-m}$ for some $m$ depending on $y$. However, the least possible $m$ is equal to 2 (or 3, say, if $M$ is a GKM-manifold); therefore, the boundary $\partial Q$ doesn't have maximal corners such as $\mathbb{R}_{\geq}^{n+1}$, $\mathbb{R} \times \mathbb{R}_{\geq}^{n}$ и $\mathbb{R}^2 \times \mathbb{R}_{\geq}^{n-1}$ (and $\mathbb{R}^3 \times \mathbb{R}_{\geq}^{n-2}$ in case of GKM-manifolds).
	
	\section{Manifolds of isospectral matrices}
	
	Denote by $\Mat_{n \times n} (\mathbb{C})$ the vector space of all complex $n \times n$ matrices.
	
	Let $\Herm(\Lambda) \subset \Mat_{n \times n} (\mathbb{C})$ be the submanifold of Hermitian matrices with fixed simple spectrum: $\Spec A = (\lambda_1, \ldots, \lambda_n)$ where $A \in \Herm(\Lambda), \Lambda = (\lambda_1, \ldots, \lambda_n)$ and $ \lambda_i \neq \lambda_j$ for $i \neq j$. It is defined by the following system of equations:
	$$
	\begin{cases} 
	tr(A) = \sum_i \lambda_i, \\ 
	tr(A^2) = \sum_i \lambda_i^2, \\ 
	... \\
	tr(A^n) = \sum_i \lambda_i^n, \\
	A^T = \overline{A};
	\end{cases}
	$$
	--- where first $n$ equations determine $\Spec(A)$. In fact, $\Herm(\Lambda)$ is diffeomorphic to the complex full flag manifold $Fl_n$ since it is the quotient of a free action of a maximal torus $T^n$ on the unitary group $U(n)$.
	
	%Кроме того полезно сделать следующее наблюдение. Рассмотрим унитарную группу $U(n)$. На ней действует тор $T^n$ справа и слева: для матрицы $A \in U(n)$ и для элемента $t = (t_1, \ldots, t_n) \in T^n$ имеем действия $A \xrightarrow{t} A T$ и $A \xrightarrow{t} T \! A$, где $T = diag(t_1, \ldots, t_n)$. Оказывается, что $Fl_n$ является фактором $U(n)$ по действию правого тора, в то время как $Herm(\Lambda)$ -- фактором по действию левого тора. То есть имеет место диаграмма:
	
	There is a natural torus action on $\Herm(\Lambda)$: for $t = (t_1, \ldots, t_n) \in T^n$ and $A \in \Herm(\Lambda)$, the action is defined by $A \xrightarrow{t} T A T^{-1}$ where $T = \diag(t_1, \ldots, t_n)$. The action of the diagonal subtorus $\Delta(T^n)$ is trivial; we have an effective action of $T^{n-1} = T^n / \Delta(T^n)$.
	
	In order to define manifolds of isospectral Hermitian matrices of staircase form $M_h$, we need to give the following definition.
	
	\begin{defin} A function $h: [n] \to [n]$ is called a Hessenberg function if it satisfies the following conditions:
		
		$\bullet$ $h(i) \geq i$ for $i = 1, \ldots, n$;
		
		$\bullet$ $h(i+1) \geq h(i)$ for $i = 1, \ldots, n-1$.
	\end{defin}
	
	For a fixed Hessenberg function $h$, we define $M_h \subset \Herm(\Lambda)$ by the equations: $a_{ij} = 0$ for $j > h(i)$ (note that since we consider Hermitian matrices, $a_{ij} = 0$ implies $a_{ji} = 0$).
	
	\begin{ex} For $h = (2, 3, 3, 6, 6, 6)$, the manifold $M_h$ consists of matrices of the following form:
		$$\begin{pmatrix} * & * & 0 & 0 & 0 & 0\\ * & * & * & 0 & 0 & 0\\ 0 & * & * & 0 & 0 & 0\\ 0 & 0 & 0 & * & * & *\\0 & 0 & 0 & * & * & *\\0 & 0 & 0 & * & * & * \end{pmatrix} 
		$$
	\end{ex}
	
	In \cite{AB1}, it was shown that for any Hessenberg function $h$ the corresponding manifold $M_h$ is smooth and its dimension $\dim_{\mathbb{R}} M_h = \sum_{i} 2 (h(i) - i)$. Since equations $a_{ij} = 0$ are invariant under the action of $T^{n-1}$ on $\Herm(\Lambda)$, we have a natural $T^{n-1}$-action on $M_h$. Its complexity $d = \sum_{i} (h(i) - i) - n$. The fixed points of the action are diagonal matrices with spectrum $\Lambda$, i.e. the matrices of the form $\diag(\lambda_{\sigma(1)}, \ldots, \lambda_{\sigma(n)})$ for all possible permutations $\sigma$.
	
	For $h = (2, 3, \ldots, n-1, n, n)$, $M_h$ is the manifold of Hermitian tridiagonal isospectral matrices; its real part is the Tomei manifold \cite{T}. In this case, the action is half-dimensional and $M_h$ is a quasitoric manifold over the permutohedron (see \cite{DJ}). %Isospectral manifolds of special form were studied, for example, in \cite{P1}, \cite{P2}.
	
	We consider manifolds $M_h$ with $d = 1$. Moreover, we consider only irreducible manifolds $M_h$, i.e. $h(i) > i$. This is possible only if  $h(i) = i+1$ for any $i$ with the exception of $i_0 \in \{1, \ldots, n-2 \}$ for which $h(i_0) = i_0+2$. In the case, the manifold $M_h$ consists of "almost tridiagonal" matrices -- that is, there is only one non-zero element over tridiagonal part, namely $a_{i_0 , i_0+2}$.
	
	We now prove that if $n \geq 4$ then the torus action on $M_h$ is not in general position. Indeed, let us take a fixed point $\diag(\lambda_{\sigma(1)}, \ldots, \lambda_{\sigma(n)})$ where $\sigma$ is a permutation. It is easy to show that the weights $\alpha_{ij}$ are of the form $e_i - e_j$ for $i < j$ such that $j \leq h(i)$ ($e_k$ are standard basis vectors). There exists $i_0$ such that $a_{i_0 , i_0+2} \neq 0$; therefore, $a_{i_0 , i_0+1} \neq 0$ and $a_{i_0+1 , i_0+2} \neq 0$ since $h$ is a Hessenberg function. Thus, there is a relation between the corresponding weights:	$$- \alpha_{i_0 , i_0+2} + \alpha_{i_0 , i_0+1} + \alpha_{i_0+1 , i_0+2} = 0.$$
	
	Hence, the weights are in general position if and only if $\alpha_{i_0 , i_0+2}, \alpha_{i_0 , i_0+1}, \alpha_{i_0+1 , i_0+2}$ are the only weights of tangent representation. This corresponds to the case $n = 3$ and $M_h = Herm(\Lambda) \cong Fl_3$. In the case, weights are in general position and the orbit space is homeomorphic to the sphere $S^4$. We proved the following proposition.
	
	\begin{prop}
		For $n \geq 4$, the action $T^{n-1} \curvearrowright M_h$ is not in general position: for any fixed point, the weights are not in general position. 
	\end{prop}
	
	Therefore, the orbit space $M_h / T$ is a manifold with boundary. It is easy to prove the following generalized statement concerning arbitrary Hessenberg functions.
	
	\begin{rem}
		For $n \geq 4$ and a Hessenberg function $h$, the action $T^{n-1} \curvearrowright M_h$ is in general position if and only if $h(i) = i+1$ for $i \in \{1, \ldots, n-1\}$ (i.e. $M_h$ is the submanifold of tridiagonal matrices).
	\end{rem}
	
	Let us consider the case $h = (3, 3, 4, 4)$. The manifold $M_h$ is the manifold of isospectral Hermitian matrices of the form:
	$$ \begin{pmatrix}	a_1 & b_{12} & b_{13} & 0 \\ 
	\overline{b}_{12} & a_2 & b_{23} & 0 \\ 
	\overline{b}_{13} & \overline{b}_{23} & a_3 & b_{34} \\ 
	0 & 0 & \overline{b}_{34} & a_4 \end{pmatrix} $$
	Lew us show that the boundary of the orbit space $\partial(M_h/T) \cong \sqcup_{i = 1}^4 S^4$.
	
	Indeed, for any fixed point, there are four weights $\alpha_{12}, \alpha_{13}, \alpha_{23}, \alpha_{34}$. Since there is a linear relation on $\alpha_{12}, \alpha_{13}, \alpha_{23}$, it follows from Theorem 2.2 that a neighborhood of a boundary point in $\partial Q$ looks like $\mathbb{R}^4 \times \mathbb{R}_{\geq}^1$ (in this case, the structure of a manifold with corners is that of a manifold with boundary).
	
	Locally, the action is modeled as the action $T^2 \times T^1 \curvearrowright \mathbb{C}^3 \times \mathbb{C}^1$ where $T^2  \curvearrowright \mathbb{C}^3$ is in general position and $T^1 \curvearrowright \mathbb{C}^1$ is standard. Therefore the condition $St(V_{G, i}) \cap T^1 \neq 0$ defines for which elements $V_{G, i}$ their orbits are boundary points. It is easy to see that such elements are determined by $b_{34} = 0$, i.e. matrices of the form:
	$$ \begin{pmatrix}	a_1 & b_{12} & b_{13} & 0 \\ 
	\overline{b}_{12} & a_2 & b_{23} & 0 \\ 
	\overline{b}_{13} & \overline{b}_{23} & a_3 & 0 \\ 
	0 & 0 & 0 & \lambda_i \end{pmatrix} $$
	--- where $\lambda_i$ is an eigenvalue. We have four elements $V_{T^1,i}$ $i = 1, \ldots, 4$ (for each possible $\lambda_i$). Note that each $V_{T^1,i}$ is equivariantly diffeomorphic to the complex full flag manifold $Fl_3$; therefore, $V_i/T \cong Fl_3/T \cong S^4$. Since $V_{T^1,i}$ are disjoint, the boundary $\partial(M_h/T)$ is homeomorphic to the disjoint sum of four spheres $S^4$.
	
	In Section 5, we will show that the orbit space $M_h/T$ is homeomorphic to $S^5 \setminus \nolinebreak (\bigsqcup_{i=1}^4 D^5)$, i.e. the complement of four disjoint open 5-disks in the 5-sphere. In order to do that, we will have to consider Hessenberg varieties and Hamiltonian torus actions.
	
	\section{Hamiltonian actions of complexity 1}
	
	Consider an action of $T = T^k$ on a smooth manifold $M = M^{2n}$ with a symplectic form $\omega$ (see, for example, \cite{ET}). We suppose the action is Hamiltonian, i.e. there is a moment map $\mu: M \to \mathfrak{t}^* \cong \mathbb{R}^k$ and the following equations hold:
	$$d \mu^\eta = - \iota(\eta_M) \omega$$
	where $\mathfrak{t}$ is the Lie algebra of $T^k$; for $\eta \in \mathfrak{t}$, $\eta_M$ is the generated vector field $\dfrac{d}{dt}|_{t=0} \ exp(t \eta)$; $\mu^\eta (x) =  \langle\mu(x), \eta\rangle$ (the pairing between $\mathfrak{t}$ and $\mathfrak{t}^*$) and $\iota$ is the contraction of $\omega$ with a vector field. For complexity $1$ Hamiltonian actions, the following theorem was proved in \cite{KT}:
	
	\begin{thm}
		Suppose $T^{n-1} \curvearrowright M^{2n}$ is a Hamiltonian action in general position; then the orbit space $M^{2n} / T^{n-1}$ is homeomorphic to the sphere $S^{n+1}$.
	\end{thm} 
	
	The actions $T^3 \curvearrowright G_{4, 2}$ and $T^2 \curvearrowright Fl_3$ studied in \cite{BT1}, \cite{BT2}, \cite{BT3} are examples of Hamiltonian actions in general position. In Section 5, we will consider Hessenberg varieties which possess Hamiltonian torus actions not in general position. For such actions, we prove the theorem.
	
	\begin{thm}  
		Suppose $T^{n-1} \curvearrowright M^{2n}$ is a Hamiltonian action in general position; then the orbit space $Q = M^{2n} / T^{n-1}$ is homeomorphic to $S^{n+1} \setminus (U_1 \sqcup \ldots \sqcup U_l)$, i.e. the complement of disjoint open domains $U_i$ in the (n+1)-sphere.
	\end{thm}
	\begin{proof}
		Since the moment map is constant on $T$-orbits, it induces the map $\overline \mu : Q \to P \subset \mathbb{R}^{n-1}$ ($P = \mu(M^{2n})$ is the moment polytope and $\pi$ is the canonical projection):
		
		\begin{center}
			\begin{tikzpicture}[scale=1.0]
			
			\coordinate (A1) at (-1,0);
			\coordinate (A2) at (1,0);						
			\coordinate (A3) at (0, -1);
			\coordinate (B1) at (0,0.1);
			\coordinate (B2) at (0.65,-0.65);		
			\coordinate (B3) at (-0.65,-0.65);
			
			\draw [thick,latex-latex] [->](-0.8,-0.05)--(0.7,-0.05) node[right]{};
			\draw [thick,latex-latex] [->](0.15,-0.85)--(0.85,-0.15) node[right]{};
			\draw [thick,latex-latex] [->](-0.9,-0.15)--(-0.1,-0.85) node[right]{};
			
			\node at (barycentric cs:A1=1) {$M$};
			\node at (barycentric cs:A2=1) {$P$};
			\node at (barycentric cs:A3=1) {$Q$};						
			\node at (barycentric cs:B1=1) {\small $\mu$};
			\node at (barycentric cs:B2=1) {\small $\overline \mu$};
			\node at (barycentric cs:B3=1) {\small $\pi$};		
			
			\end{tikzpicture}
		\end{center}
		
		In \cite{KT}, it was shown that $\overline \mu ^{-1} (y)$ for $y \in P$ is either a point or a 2-sphere and the conditions hold:
		
		$\bullet$ for an interior point $y \in \Int P$, the level set $\overline \mu ^{-1} (y)$ is a 2-sphere;
		
		$\bullet$ for points $y, y'$ from the relative interior of a face of $P$, the level sets $\overline \mu ^{-1} (y)$ and $\overline \mu ^{-1} (y')$ coincide.
		
		Therefore it is possible to construct a model for $Q \cong (P \times S^2) / \!\! \sim$ where $\sim$ is the relation which collapses the 2-spheres over some faces of $P$. For actions in general position, the level set $\overline \mu ^{-1} (y)$ is a point for any $y \in \partial P$. In the case the model $(P \times S^2) / \!\! \sim \ \cong (D^{n-1} \times S^2) / \!\! \sim$ is homeomorphic to the sphere $S^{n+1}$.
		
		For actions not in general position, we have that there are some facets $F_1, \ldots, F_m$, which we call special facets, such that the level set $\overline \mu ^{-1} (y)$ is a 2-sphere for $y \in F_i$. It is easy to show explicitly that the boundary of the orbit space $\partial Q =  \bigcup_i (F_i \times S^2)/ \!\! \sim$. Note that special facets are not necessarily disjoint (as we will see in Section 5 for Hessenberg varieties).
		
		Denote by $B_1, \ldots, B_l$ the connected components of the union of special facets $F_1 \cup \ldots \cup F_m$; the model $(P \times S^2) / \!\! \sim$ is homeomorphic to $S^{n+1} \setminus (U_1 \sqcup \ldots \sqcup U_l)$ where $U_i = B_i \times D^3$ is an open domain (here we assume that $D^3$ is the open 3-disk). The boundary of the orbit space $\partial Q = \bigsqcup_i \partial U_i$ where $\partial U_i = B_i \times S^2 / \!\! \sim$.
	\end{proof}

	\section{Regular Hessenberg varieties}
	
	Let us fix a linear transformation $A: \mathbb{C}^n \to \mathbb{C}^n$ and a Hessenberg function $h$. A Hessenberg variety $\Hess(A, h)$ defined by $A$ and $h$ is a subvariety of the full complex flag variety such that 	$$\Hess(A, h) = \{V_{\bullet} = (V_0, \ldots, V_n) \in Fl_n | AV_{i} \subset V_{h(i)} \}.$$
	
	Hessenberg varieties are studied in algebraic geometry (see, for example, \cite{AFZ}, \cite{AC}); in \cite{AB1}, it was shown that $H_h = \Hess(\Lambda, h)$ where $\Lambda = \diag(\lambda_1, \ldots, \lambda_n)$, $\lambda_i \neq \lambda_j$ for $i \neq j$ is, in a sense, dual to $M_h$, which we defined in Section 3. In particular, their GKM-graphs (the 1-skeleton of the partition by orbit types) and orbit spaces coincide. For illustrative purposes, we will describe partition elements $V_{G, i}$ for $H_h$ in terms of submanifolds in $M_h$.
	
	We will use the technique we discussed in the proof of Theorem 4.2 and describe the orbit spaces of $H_h$ using the moment polytope (since $H_h$ is a projective variety, it is symplectic and possesses a Hamiltonian torus action). For any $h$, the moment polytope for $H_h$ is the $(n-1)$-permutohedron.
	
	Note, however, that while $H_h$ are symplectic submanifolds of the full complex flag variety $Fl_n$, the restriction of the symplectic form from $\Herm(\Lambda)$ to $M_h$ is degenerate (see \cite{AB1}, \cite{AB2}; it follows from the fact that their moment polytopes are not convex).
	
	Therefore, the existence of an almost moment map (in the sense of \cite{BT1}) is a non-trivial fact given by the diagram:
	
	\begin{center}
		\begin{tikzpicture}[scale=1.0]
		
		\coordinate (A1) at (0,1);
		\coordinate (A2) at (-1.5,0);						
		\coordinate (A3) at (1.5,0);
		\coordinate (A4) at (0,-1);	
		\coordinate (A5) at (-1.5,-2);				
		
		\coordinate (B1) at (-0.9,-0.35);
		\coordinate (B2) at (0.9,-0.35);
		\coordinate (B3) at (-1.65,-1);
		\coordinate (B4) at (-0.9,-1.65);
		\coordinate (B5) at (-1,0.7);
		\coordinate (B6) at (1.05,0.75);
		
		\draw [thick,latex-latex] [->](-0.4,0.8)--(-1.35,0.15) node[right]{};
		\draw [thick,latex-latex] [->](0.4,0.8)--(1.3,0.2) node[right]{};
		\draw [thick,latex-latex] [->](-1.4,-0.2)--(-0.75,-0.75) node[right]{};
		\draw [thick,latex-latex] [->](1.4,-0.2)--(0.75,-0.75) node[right]{};
		\draw [thick,latex-latex] [->](-1.5,-0.2)--(-1.5,-1.8) node[right]{};
		\draw [thick,latex-latex] [->](-0.75,-1.25)--(-1.35,-1.8) node[right]{};
		
		\node at (barycentric cs:A1=1) {$U(n)$};
		\node at (barycentric cs:A2=1) {$H_h$};
		\node at (barycentric cs:A3=1) {$M_h$};
		\node at (barycentric cs:A4=1) {$H_h / T \cong M_h / T$};
		\node at (barycentric cs:A5=1) {$Pe^{n-1}$};	
		\node at (barycentric cs:B1=1) {\small $\pi_H$};
		\node at (barycentric cs:B2=1) {\small $\pi_M$};
		\node at (barycentric cs:B3=1) {\small $\mu$};
		\node at (barycentric cs:B4=1) {\small $\overline{\mu}$};
		\node at (barycentric cs:B5=1) {\small $/T$};
		\node at (barycentric cs:B6=1) {\small $/T$};
		
		\end{tikzpicture}
	\end{center}
	--- the map is given by the composition $\overline{\mu} \circ \pi_M$ of the canonical projection $\pi_M$ and the orbital moment map $\overline{\mu}$ for $H_h$. Recall that both $Fl_n$ and $\Herm(\Lambda)$, which contain $H_h$ и $M_h$ respectively, are orbit spaces of actions of maximal tori on the unitary group $U(n)$.
	
	$\bullet$ Consider the case $h = (3,3,4,4)$. The fixed points are diagonal matrices with entries $\lambda_1, \ldots, \lambda_4$ (due to the spectrum condition). There are $4! = 24$ such matrices and they correspond to the vertices of the 3-permutohedron. 
	
	Partition elements with 2-dimensional stabilizers are given by matrices of the following types:
	\begin{center}
		$\begin{pmatrix}
		a_1 & b_{12} & 0 & 0 \\ \overline{b}_{12} & a_2 & 0 & 0 \\ 0 & 0 & \lambda_i & 0 \\ 0 & 0 & 0 & \lambda_j
		\end{pmatrix},$ 
		$\begin{pmatrix}
		\lambda_i & 0 & 0 & 0 \\ 0 & a_2 & b_{23} & 0 \\ 0 & \overline{b}_{23} & a_3 & 0 \\ 0 & 0 & 0 & \lambda_j
		\end{pmatrix},$
		
		$\begin{pmatrix}
		\lambda_i & 0 & 0 & 0 \\ 0 & \lambda_j & 0 & 0 \\ 0 & 0 & a_3 & b_{34} \\ 0 & 0 & \overline{b}_{34} & a_4
		\end{pmatrix},$
		$\begin{pmatrix}
		a_1 & 0 & b_{13} & 0 \\ 0 & \lambda_i & 0 & 0 \\ \overline{b}_{13} & 0 & a_3 & 0 \\ 0 & 0 & 0 & \lambda_j
		\end{pmatrix}.$
	\end{center}
	
	Each possible type gives 12 partition elements (there are 12 possibilities for a choice of $\lambda_i, \lambda_j$). The first three types correspond to the edges of the permutohedron; the last type represents additional edges. They are admissible 1-polytopes (in the sense of \cite{BT1}, \cite{BT2}), i.e. the image of corresponding elements under the moment map is a subpolytope (the convex hull of some vertices of the moment polytope) which is not a face of the moment polytope.
	
	The elements corresponding to the described types represent all invariant 2-spheres, therefore we have the GKM-graph associated to $M_h$ (see Fig. 1). Its vertices are labeled with 4-permutations; for instance, the vertex $(1234)$ corresponds to the matrix $\diag(\lambda_1, \lambda_2, \lambda_3, \lambda_4)$. The edges of the permutohedron correspond to the transpositions $(12), (23), (34)$; for instance, the edge between the vertices $(1234)$ and $(1324)$ is given by the element of following matrices:
	$$\begin{pmatrix}	\lambda_1 & 0 & 0 & 0 \\ 
	0 & a_2 & b_{23} & 0 \\ 
	0 & \overline{b}_{12} & a_3 & 0 \\ 
	0 & 0 & 0 & \lambda_4
	\end{pmatrix}$$
	It is easy to show that due to the spectrum condition, the element is a 2-sphere touching the points $\diag(\lambda_1, \lambda_2, \lambda_3, \lambda_4)$ and $\diag(\lambda_1, \lambda_3, \lambda_2, \lambda_4)$.

	\begin{figure}
		\begin{center}
			\includegraphics[scale=0.7]{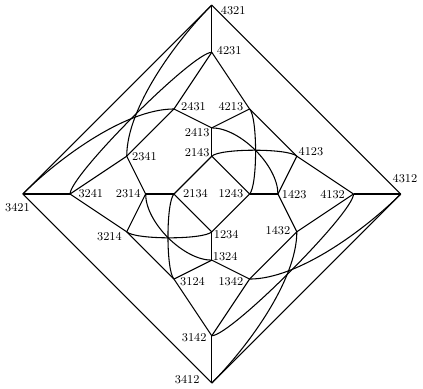}
			\caption{GKM-graph associated to $M_{(3,3,4,4)}$.}	
		\end{center}   
	\end{figure} 	
	
	%	\begin{figure}
	%		\begin{center}
	%		\includegraphics[width=0.5\textwidth]{3-Permutohedron}
	%		
	%		\caption{ГКМ-граф для $M_{(3,3,4,4)}$.}	
	%		\end{center}
	%	\end{figure}
	
	Analogously, the matrices of the last type correspond to the remaining edges, i.e. the edges corresponding to the transposition $(13)$. We have 3 additional edges in 4 mutually disjoint hexagons. In fact, each of the hexagons along with 3 diagonal edges is the GKM-graph associated to a submanifold equivariantly diffeomorphic to the full complex flag manifold $Fl_3$ (see the example after Remark 3.4).
	
	We now describe special facets of the moment polytope. The following matrices:
	
	\begin{center}
		$\begin{pmatrix} a_1 & b_{12} & b_{13} & 0\\ {\overline b_{12}} & a_2 & b_{23} & 0\\ {\overline b_{13}} & {\overline b_{23}} & a_3 & 0\\0 & 0 & 0 & \lambda_j \end{pmatrix},
		\begin{pmatrix} a_1 & b_{12} & b_{13} & 0\\ {\overline b_{12}} & a_2 & 0 & 0\\ {\overline b_{13}} & 0 & a_3 & b_{34}\\0 & 0 & {\overline b_{34}} & a_4 \end{pmatrix}$
		
		$\begin{pmatrix} a_1 & b_{12} & 0 & 0\\ {\overline b_{12}} & a_2 & b_{23} & 0\\ 0 & {\overline b}_{23} & a_3 & b_{34}\\0 & 0 & {\overline b}_{34} & a_4 \end{pmatrix},
		\begin{pmatrix} a_1 & 0 & b_{13} & 0\\ 0 & a_2 & b_{23} & 0\\ {\overline b_{13}} & {\overline b_{23}} & a_3 & b_{34}\\0 & 0 & {\overline b_{34}} & a_4 \end{pmatrix}$
	\end{center}
	
	--- correspond to elements of dimension 6. However, the dimension of effective torus action equals 2 for the first type and 3 for the other types. Therefore, the special facets are given by the elements consisting of the matrices of the first type. There are 4 special facets --- 4 mutually disjoint hexagons; the corresponding element, as it was noted before, are submanifolds equivariantly diffeomorphic to the full complex flag manifold $Fl_3$. Thus, the orbit space $$M_h / T^3 \cong S^5 \setminus (\bigsqcup_{i=1}^4 D^5).$$
	
	Note that for $n=4$, there is another case of Hessenberg varieties $H_h$ (and, respectively, manifolds of isospectral matrices $M_h$) where the torus action is of complexity 1. It is the symmetrical case $h = (2, 4, 4, 4)$ (the symmetry is provided by the transposition of matrices with respect to the anti-diagonal). For $n=5$, we have two different cases $h = (3, 3, 4, 5, 5)$ и $h = (2, 4, 4, 5, 5)$ (the first one is symmetrical to the case $h = (2, 3, 5, 5, 5)$); we will see that the orbit spaces have different topology.
	
	Before studying the orbit spaces $H_h$ and $M_h$ for $n=5$, we will describe the partition elements corresponding to special facets for arbitrary $n$. 
	
	Nore that the permutohedron $Pe^n$ admits a natural $n$-coloring of facets in the following way. Any face of $Pe^n$ is a product of permutohedra of lower dimensions; naturally, facets are the products $Pe^{k-1} \times Pe^{n-k}$. One can color a facet of the form $Pe^{k-1} \times Pe^{n-k}$ with the color $k$.
	
	It is possible to describe faces of the $n$-permutohedron using the manifold of tridiagonal Hermitian $(n+1)\times(n+1)$ matrices (it is a quasitoric manifold over the permutohedron), i.e. the submanifold $M_{(2, 3, \ldots, n-1,n,n)}$ of matrices of the following form:	
	\begin{center}
		$\begin{pmatrix}
		a_1 & b_{12} & 0 & \ldots & 0 & 0\\ \overline{b}_{12} & a_2 & b_{23} & \ldots & 0 & 0\\ 0 & \overline{b}_{23} & a_3 & \ldots & 0 & 0\\ \ldots & \ldots & \ldots & \ldots & \ldots & \ldots\\ 0 & 0 & 0 & \ldots & a_n & b_{n,n+1} \\ 0 & 0 & 0 & \ldots & \overline{b}_{n,n+1}  & a_{n+1}
		\end{pmatrix}$
	\end{center}
	It is easy to see that facets of color $k$ are submanifolds of $M_{(2, 3, \ldots, n-1,n,n)}$ given by $b_{k,k+1} = \nolinebreak 0$. Since we have block matrices in the case, such submanifold is the product of two submanifolds of tridiagonal matrices of smaller dimension. There are $\binom{n}{k}$ facets of color $k$ (that is, the possible choices of $k$ eigenvalues to be the spectrum of the upper $k \times k$ block).
	
	Consider now a Hessenberg function $h: [n+1] \to [n+1]$ such that $h(i) = i+1$ for all $i$ except of $i_0 \in \{1, \ldots, n-1 \}$ for which $h(i_0) = i_0+2$. The corresponding Hessenberg variety $H_h$ and isospectral submanifold $M_h$ are of dimension $2n+2$; the dimension of тор, действующий эффективно, имеет размерность $n$. Образом отображения моментов является пермутоэдр $Pe^n$.
	
	We now describe the special facets in this case. It is easy to see that the elements corresponding to the special facets are among those given by one equation $b_{i, i+1} = 0$ for $i \in [n]$ or $b_{i_0, i_0+2} = 0$ (otherwise the dimension of the elements implies that they correspond to faces of codimension $\geq 2$). It follows easily that only for $b_{i_0, i_0+1} = 0$, $b_{i_0+1, i_0+2} = 0$ and $b_{i_0, i_0+2} = 0$, the effective torus action is half-dimensional and the elements given by the other equations correspond to special facets; therefore, the following proposition is true.
	
	\begin{prop}
		For Hessenberg varieties with torus actions of complexity 1 (defined by Hessenberg functions $h$ described above), the special facets of the moment polytope $Pe^n$ are exactly facets of color $k$ for $k \neq i_0, i_0+1$.
	\end{prop}
	
	For the case $n = 5$, we will explicitly show the topology of the orbit spaces.
	
	$\bullet$ Let $h = (3, 3, 4, 5, 5)$; the corresponding manifold $M_h$ consists of matrices of the following form:
	
	\begin{center}
		$\begin{pmatrix}
		a_1 & b_{12} & b_{13} & 0 & 0\\ \overline{b}_{12} & a_2 & b_{23} & 0 & 0\\ \overline{b}_{13} & \overline{b}_{23} & a_3 & b_{34} & 0 \\ 0 & 0 & \overline{b}_{34} & a_4 & b_{45} \\ 0 & 0 & 0 & \overline{b}_{45} & a_5
		\end{pmatrix}$
	\end{center}
	
	The special facets correspond to the elements of matrices of the following types:
	
	\begin{center}
		$$A_{3,3,3,5,5}:\begin{pmatrix}
		a_1 & b_{12} & b_{13} & 0 & 0\\ \overline{b}_{12} & a_2 & b_{23} & 0 & 0\\ \overline{b}_{13} & \overline{b}_{23} & a_3 & 0 & 0 \\ 0 & 0 & 0 & a_4 & b_{45} \\ 0 & 0 & 0 & \overline{b}_{45} & a_5
		\end{pmatrix},$$
		$$A_{3,3,4,4,5}:\begin{pmatrix}
		a_1 & b_{12} & b_{13} & 0 & 0\\ \overline{b}_{12} & a_2 & b_{23} & 0 & 0\\ \overline{b}_{13} & \overline{b}_{23} & a_3 & b_{34} & 0 \\ 0 & 0 & \overline{b}_{34} & a_4 & 0 \\ 0 & 0 & 0 & 0 & \lambda_i
		\end{pmatrix}.$$
	\end{center}
	
	The elements of the type $A_{3,3,3,5,5}$ correspond to the facets of color $3$; such facets are hexagonal prisms $Pe^2 \times Pe^1$ since $Pe^2$ is a hexagon and $Pe^1$ is a line segment. The elements of the type $A_{3,3,4,4,5}$ correspond to the facets of color $4$; such faces are permutohedra $Pe^3$. There are 10 facets of color $3$ and 5 facets of color $4$.
	
	We will prove that the boundary of the orbit space consists of one connected component. Indeed, a facet of color $4$, corresponding to the element of matrices of the type $A_{3,3,4,4,5}$ with $\lambda_i$, intersects four facets of color $3$, corresponding to the element of matrices of the type $A_{3,3,3,5,5}$ such that the lower block $\begin{pmatrix} a_4 & b_{45} \\ \overline{b}_{45} & a_5 \end{pmatrix}$ has eigenvalue $\lambda_i$.
	
	From the point of view of the orbit space, we have the 5-sphere without 4 disks over each facet of color $4$ since each element of the type $A_{3,3,4,4,5}$ is equivariantly diffeomorphic to the Hessenberg variety $H_{(3, 3, 4, 4)}$. Over each facet of color $3$ we have the $"$tube$"$ $S^4 \times [0,1]$. As it was shown earlier, each facet of color $4$ intersects four facets of color $3$; hence, each tube connects the punctured 5-spheres with each other (the intersection of a tube with a punctured sphere is the boundary component $S^4$).
	
	Therefore, we have the complete graph $K_5$: its vertices are the punctured 5-spheres and its edges are the connecting tubes. The vertex labeled with the number $i$ corresponds to the element of the type $A_{3,3,4,4,5}$ with $\lambda_i$; the edge connecting vertices $i$ and $j$ corresponds to the element of the type $A_{3,3,3,5,5}$ such that the lower block $\begin{pmatrix} a_4 & b_{45} \\ \overline{b}_{45} & a_5 \end{pmatrix}$ has eigenvalues $\lambda_i$ and $\lambda_j$.
	
	%\begin{figure}
	%		\begin{center}
	%			\begin{tikzpicture}[scale=1.0]
	
	%			\coordinate (A1) at (0,1.2);
	%			\coordinate (A2) at (-1.65,0);						
	%			\coordinate (A3) at (-1.1,-1.55);
	%			\coordinate (A4) at (1.1,-1.55);	
	%			\coordinate (A5) at (1.65,0);				
	
	%			\draw [thick,latex-latex] [-](0,1)--(-1.5,0) node[right]{};
	%			\draw [thick,latex-latex] [-](-1.5,0)--(-1,-1.4) node[right]{};
	%			\draw [thick,latex-latex] [-](-1,-1.4)--(1,-1.4) node[right]{};
	%			\draw [thick,latex-latex] [-](1,-1.4)--(1.5,0) node[right]{};
	%			\draw [thick,latex-latex] [-](1.5,0)--(0,1) node[right]{};
	
	%			\draw [thick,latex-latex] [-](0,1)--(-1,-1.4) node[right]{};
	%			\draw [thick,latex-latex] [-](-1,-1.4)--(1.5,0) node[right]{};
	%			\draw [thick,latex-latex] [-](1.5,0)--(-1.5,0) node[right]{};
	%			\draw [thick,latex-latex] [-](-1.5,0)--(1,-1.4) node[right]{};
	%			\draw [thick,latex-latex] [-](1,-1.4)--(0,1) node[right]{};	
	
	%			\node at (barycentric cs:A1=1) {$1$};
	%			\node at (barycentric cs:A2=1) {$2$};
	%			\node at (barycentric cs:A3=1) {$3$};
	%			\node at (barycentric cs:A4=1) {$4$};
	%			\node at (barycentric cs:A5=1) {$5$};
	
	%			\end{tikzpicture}
	%		\caption{Полный граф $K_5$.}
	%		\end{center} 
	%\end{figure}	
	
	Therefore, the orbit space $Q_{(3,3,4,5,5)} = M_{(3,3,4,5,5)} / T^4$ has only one boundary component $\#_{K_5} S^5$, the connected sum of five spheres $S^5$ over the graph $K_5$. The orbit space is homeomorphic to $$S^6 \setminus (\#_{K_5} D^6)$$ --- the complement of the connected sum of five open 6-disks in the 6-sphere.
	
	$\bullet$ 
	Let $h = (2, 4, 4, 5, 5)$; the corresponding manifold $M_h$ consists of matrices of the following form:
	
	\begin{center}
		$\begin{pmatrix}
		a_1 & b_{12} & 0 & 0 & 0\\ \overline{b}_{12} & a_2 & b_{23} & b_{24} & 0\\ 0 & \overline{b}_{23} & a_3 & b_{34} & 0\\ 0 & \overline{b}_{24} & \overline{b}_{34} & a_4 & b_{45} \\ 0 & 0 & 0 & \overline{b}_{45} & a_5
		\end{pmatrix},$
	\end{center}
	
	The special facets correspond to the elements of matrices of the following types:
	
	\begin{center}
		$$A_{1,4,4,5,5}:\begin{pmatrix}
		\lambda_i & 0 & 0 & 0 & 0\\ 0 & a_2 & b_{23} & b_{24} & 0\\ 0 & \overline{b}_{23} & a_3 & b_{34} & 0\\ 0 & \overline{b}_{24} & \overline{b}_{34} & a_4 & b_{45} \\ 0 & 0 & 0 & \overline{b}_{45} & a_5
		\end{pmatrix},$$
		$$A_{2,4,4,4,5}\begin{pmatrix}
		a_1 & b_{12} & 0 & 0 & 0\\ \overline{b}_{12} & a_2 & b_{23} & b_{24} & 0\\ 0 & \overline{b}_{23} & a_3 & b_{34} & 0\\ 0 & \overline{b}_{24} & \overline{b}_{34} & a_4 & 0 \\ 0 & 0 & 0 & 0 & \lambda_i
		\end{pmatrix}.$$
	\end{center}
	
	The elements of both types correspond to the facets of color $1$ and $4$; such facets are permutohedra $Pe^3$. There are 10 facets of color $1$ and 10 facets of color $4$. A facet of color $1$, corresponding to the element of matrices of the type $A_{1,4,4,5,5}$ with $\lambda_i$, intersects four facets of color $4$, corresponding to the element of matrices of the type $A_{2,4,4,4,5}$ with $\lambda_j$ for $j \neq i$.
	
	From the point of view of the orbit space, we have the 5-sphere without 4 disks over each facet since each element of both types is equivariantly diffeomorphic to the Hessenberg variety $H_{(3, 3, 4, 4)}$. The intersection of two punctured spheres is the boundary component $S^4$.
	
	We have the almost complete bipartite graph $\tilde{K}_{5,5}$ (see Fig. 2): the vertex, say, $1a$ is connected with all vertices from the part $b$ except of $1b$. The vertex $1a$ (respectively, $1b$) corresponds to the elements of the type $A_{1,4,4,5,5}$ (respectively, $A_{2,4,4,4,5}$) with $\lambda_1$.
	
	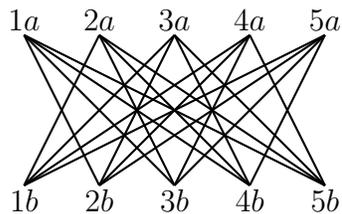
\begin{figure}
		\begin{center}
			\begin{tikzpicture}[scale=1.0]
			
			\coordinate (A1) at (-2,-1.2);
			\coordinate (A2) at (-1,-1.2);						
			\coordinate (A3) at (0,-1.2);
			\coordinate (A4) at (1,-1.2);	
			\coordinate (A5) at (2,-1.2);				
			
			\coordinate (B1) at (-2,1.2);
			\coordinate (B2) at (-1,1.2);						
			\coordinate (B3) at (0,1.2);
			\coordinate (B4) at (1,1.2);	
			\coordinate (B5) at (2,1.2);	
			
			\draw [thick,latex-latex] [-](-2,-1)--(-1,1) node[right]{};
			\draw [thick,latex-latex] [-](-2,-1)--(0,1) node[right]{};
			\draw [thick,latex-latex] [-](-2,-1)--(1,1) node[right]{};
			\draw [thick,latex-latex] [-](-2,-1)--(2,1) node[right]{};
			
			\draw [thick,latex-latex] [-](-1,-1)--(-2,1) node[right]{};
			\draw [thick,latex-latex] [-](-1,-1)--(0,1) node[right]{};
			\draw [thick,latex-latex] [-](-1,-1)--(1,1) node[right]{};
			\draw [thick,latex-latex] [-](-1,-1)--(2,1) node[right]{};
			
			\draw [thick,latex-latex] [-](0,-1)--(-2,1) node[right]{};
			\draw [thick,latex-latex] [-](0,-1)--(-1,1) node[right]{};
			\draw [thick,latex-latex] [-](0,-1)--(1,1) node[right]{};
			\draw [thick,latex-latex] [-](0,-1)--(2,1) node[right]{};
			
			\draw [thick,latex-latex] [-](1,-1)--(-2,1) node[right]{};
			\draw [thick,latex-latex] [-](1,-1)--(-1,1) node[right]{};
			\draw [thick,latex-latex] [-](1,-1)--(0,1) node[right]{};
			\draw [thick,latex-latex] [-](1,-1)--(2,1) node[right]{};
			
			\draw [thick,latex-latex] [-](2,-1)--(-2,1) node[right]{};
			\draw [thick,latex-latex] [-](2,-1)--(-1,1) node[right]{};
			\draw [thick,latex-latex] [-](2,-1)--(0,1) node[right]{};
			\draw [thick,latex-latex] [-](2,-1)--(1,1) node[right]{};

			\node at (barycentric cs:A1=1) {$1b$};
			\node at (barycentric cs:A2=1) {$2b$};
			\node at (barycentric cs:A3=1) {$3b$};
			\node at (barycentric cs:A4=1) {$4b$};
			\node at (barycentric cs:A5=1) {$5b$};
			
			\node at (barycentric cs:B1=1) {$1a$};
			\node at (barycentric cs:B2=1) {$2a$};
			\node at (barycentric cs:B3=1) {$3a$};
			\node at (barycentric cs:B4=1) {$4a$};
			\node at (barycentric cs:B5=1) {$5a$};
			\end{tikzpicture}
			\caption{Almost complete bipartite graph $\tilde{K}_{5,5}$.}	
		\end{center}
	\end{figure}
	
	We conclude that the boundary of the orbit space $Q_{(2,4,4,5,5)}$ consists of one connected component $\#_{\tilde{K}_{5,5}} S^5$, i.e. the connected sum of ten spheres $S^5$ over the graph $\tilde{K}_{5,5}$. The orbit space $Q_{(2,4,4,5,5)}$ is homeomorphic to $$S^6 \setminus (\#_{\tilde{K}_{5,5}} D^6).$$
	
	Now note that since each of the orbit spaces $Q_{(2,4,4,5,5)}$ and $Q_{(3,3,4,5,5)}$ is the complement of a tubular neighborhood of a graph in the 6-sphere, their homology groups are determined by the graphs due to Alexander duality. In particular, 
	$$H^4(Q_{(3,3,4,5,5)}) = H^4 (S^6 \setminus (\#_{K_5} D^6)) \cong H_1(\#_{K_5} D^6),$$
	$$H^4(Q_{(2,4,4,5,5)}) = H^4 (S^6 \setminus (\#_{\tilde{K}_{5,5}} D^6)) \cong H_1(\#_{\tilde{K}_{5,5}} D^6).$$
	Homology groups of each tubular neighborhood are those of the corresponding graphs:
	$$H_1(\#_{K_5} D^6) \cong H_1(K_5) \cong \mathbb{Z}^6,$$
	$$H_1(\#_{\tilde{K}_{5,5}} D^6) \cong H_1(\tilde{K}_{5,5}) \cong \mathbb{Z}^{11}.$$
	It also follows that cohomologies of both $Q_{(3,3,4,5,5)}$ and $Q_{(2,4,4,5,5)}$ in dimensions 1, 2 and 3 are trivial. In fact, we have the following general statement.
	\begin{thm}
		For a Hamiltonian action $T^{n-1} \curvearrowright M^{2n}$ not in general position, the reduced cohomology groups $\tilde{H}^i (M^{2n} / T^{n-1}) = 0$ for $i = 0, 1, 2$. 
	\end{thm}
	The theorem follows from Alexander duality and the proof of Theorem 4.2 where we showed that $M^{2n} / T^{n-1}$ is homeomorphic to $S^{n+1} \setminus (U_1 \sqcup \ldots \sqcup U_l)$ where $U_i = B_i \times D^3$ have dimension at most $n-2$.

\end{document}